\documentclass[12pt]{amsart}
\usepackage[]{latexsym,amssymb,amsmath,amsfonts, amsthm}
\usepackage[all,cmtip,ps]{xy}
\usepackage{enumitem}
\usepackage[dvipsnames]{xcolor}

\DeclareMathOperator{\R}{\mathbf{R}}

\def\R{\mathcal{R}}

\def\N{{\mathbb N}}

\def\Im{\operatorname{Im}}
\def\Hom{\operatorname{Hom}}

\def\dualita#1#2{\mathrel{
                 \mathop{\vcenter{
                 \offinterlineskip
                 \hbox to 0.6truecm{\rightarrowfill}
                 \hbox to 0.6truecm{\leftarrowfill}}}%
                 \limits_{#2}^{#1}}}

\DeclareMathOperator{\Ann}{Ann}

\DeclareMathOperator{\Ker}{Ker}

\newtheorem{theorem}{Theorem}

\newtheorem{corollary}[theorem]{Corollary}

\newtheorem{definition}[theorem]{Definition}

\newtheorem{lemma}[theorem]{Lemma}

\newtheorem{proposition}[theorem]{Proposition}
\theoremstyle{remark}
\newtheorem{remark}[theorem]{Remark}

\begin{document}

\title[Injective modules over  the Jacobson algebra]{Injective modules over \\ the Jacobson algebra $K\langle X, Y  \ | \ XY=1\rangle$}


\author{Gene Abrams}
\address{Department of Mathematics, University of Colorado, 1420 Austin Bluffs Parkway, 
Colorado Springs, CO 80918 U.S.A.}
\email{abrams@math.uccs.edu}

\author{Francesca Mantese}
\address{Dipartimento di Informatica, Universit\`{a} degli Studi di Verona,  strada le grazie 15, 37134  Verona, Italy}
\email{francesca.mantese@univr.it}

\author{Alberto Tonolo}
\address{Dipartimento di Matematica Tullio Levi-Civita, Universit\`{a} degli Studi di Padova, via Trieste 63,  35121, Padova, Italy}
\email{tonolo@math.unipd.it}

\begin{abstract}
For a field $K$, let $\mathcal{R}$ denote the Jacobson algebra  $K\langle X, Y  \ | \ XY=1\rangle$.   We give an explicit construction of the injective envelope of each of the (infinitely many) simple left $\mathcal{R}$-modules.  Consequently, we obtain an explicit  description of a minimal injective cogenerator for $\mathcal{R}$.   Our approach involves realizing $\mathcal{R}$ up to isomorphism as the Leavitt path $K$-algebra of an appropriate graph $\mathcal{T}$,  which thereby allows us to utilize important machinery developed for that class of algebras.   \\



\end{abstract}

\maketitle

\section{introduction}

A unital ring $A$ is called {\it directly finite} in case, for any $x,y \in A$, if $xy=1$ then $yx=1$;  equivalently, in case every element of $A$ which is invertible on one side is invertible on the other side as well.  It is not hard to show that rings which satisfy various natural conditions (obviously commutativity, but also rings which satisfy some mild chain condition, etc.) are directly finite.  On the other hand, examples  abound of rings containing elements $x,y$ for which the condition fails.  Perhaps the most natural 'concrete' example is found in the endomorphism ring of a countably-infinite-dimensional vector space $V$ over a field;  the shift transformation $y$  which takes $e_i$ to $e_{i+1}$ (where $\{e_i \ | \ i\in \N\}$ is a basis for $V$) has this property. 

A moment's reflection yields that there is an even more natural example of a ring which fails to be directly finite, to wit:   
$$ \R = K\langle X,Y  \ | \ XY=1 \rangle,$$  the free associative $K$-algebra on two (noncommuting) generators, modulo the single relation $XY=1$.  
A search of the literature suggests that  this algebra was first explicitly studied by Jacobson in the late 1940's in \cite{Jacobson}.    
As such, for a field $K$ we will throughout the article refer to this algebra as the {\it Jacobson algebra} over $K$.\footnote{Because of its close relationship to the well-studied Toeplitz $C^*$-algebra, the Jacobson $K$-algebra $\R$ has also been called the ``algebraic Toeplitz $K$-algebra" elsewhere in the literature, see e.g. \cite{TheBook}. 
We prefer to call $\R$ the ``Jacobson algebra" to further emphasize our focus on $\R$ as an algebra in and of itself, rather than on its connection to graph $C^*$-algebras.} 
While the displayed description of $\R$ is quite straightforward,  the  structure of $\R$ is anything but.      

Various ring-theoretic and module-theoretic properties of $\R$ have been analyzed during the seven decades since Jacobson's work, including in:   Cohn \cite{Cohn} (1966);  Bergman \cite{Bergman} (1974);   Gerritzen \cite{Gerritzen} (2000);   Bavula \cite{Bavula} (2010);  Ara and Rangaswamy  \cite{AR} (2014);   Iovanov and Sistko \cite{IovanovSistko} (2017); { and Lu, Wang and Wang \cite{LuWangWang} (2019).}

It is interesting to note that in a majority of these articles (specifically \cite{Cohn}, \cite{Bergman}, \cite{Bavula}, and \cite{AR}), the structural properties of $\R$ follow from the fact that $\R$ appears as the ``smallest" or ``base case" of an infinite class of algebras.   Further, it is fair to say that there has been much focus in these works on the projective module structure of the algebras, but relatively little focus on the injective modules.

For the directed graph  $\mathcal{T} = \hspace{-.25in} \xymatrix{   & {\bullet} \ar@(ul,dl) \ar[r] &{\bullet}}$,  the Jacobson algebra $\R$ is isomorphic to the Leavitt path algebra $L_K(\mathcal{T})$ (see Proposition \ref{JacobsonisLeavitt} below).  The interpretation of the Jacobson algebra as a Leavitt path algebra will guide the investigation herein.  We refer those readers who are unfamiliar with Leavitt path algebras to \cite[Chapter 1]{TheBook}.

Following \cite{AR},
there are three classes of \emph{Chen} simple modules for Leavitt path algebras $L_K(E)$ of a general (finite) graph $E$:
\begin{itemize}
\item simple modules associated to sinks;
\item simple modules associated to infinite irrational paths;
\item simple modules associated to pairs consisting of infinite rational paths together with  irreducible polynomials in $K[x]$ with constant term equal to $-1$.\end{itemize}
Further, because the unique cycle in $\mathcal{T}$ is necessarily maximal, by \cite[Corollary 4.6]{AR} we conclude that every simple left $L_K(\mathcal{T})$-module  is a Chen simple;  so a complete list of non-isomorphic simple left $L_K(\mathcal{T})$-modules is given by
\begin{itemize}
\item the Chen simple module $L_K(\mathcal{T})w$ associated to the sink $w$, and
\item the Chen simple modules $V^f$ associated to the infinite rational path $c^\infty$ (where $c$ is the loop in $\mathcal{T}$), and to irreducible polynomials $f(x)$ in $K[x]$ with $f(0)=-1$.
\end{itemize}

The results presented in the current work, when placed  in the context of  previous collaborative work of the three authors (\cite{AMTExt}, \cite{AMTBezout}, \cite{AMTPrufer}),  follow the same ``smallest case" approach as that taken in the previously noted articles.     Among other things,  results regarding: the ${\rm Ext}^1$ groups of pairs of Chen simple modules; the B\'{e}zout property;  the construction of ``Pr\"{u}fer-like" modules for Chen simple modules;  and the construction of injective envelopes for some of these Chen simples,  have been achieved in these three papers for Leavitt path algebras $L_K(E)$ of general (finite) graphs $E$.   In the current work  we bundle  some of  the consequences of these results together with a new type of construction in the specific case where $E = \mathcal{T} $.  When viewed in the more general landscape, the graph $\mathcal{T} $ can be viewed as the smallest nontrivial example in which, on the one hand,   various of the known general Leavitt path algebra results may be applied, while, on the other hand,  in which a general construction of the injective envelopes of simple modules is heretofore not known.

 More to the point, we have constructed in \cite{AMTPrufer}  the injective envelopes  of the Chen simple modules which correspond to maximal cycles in a graph $E$, for the specific irreducible polynomial $f(x)=x-1$.  Our first task here is to show that a similar construction yields a description of the injective envelope of any Chen simple module which corresponds to the maximal cycle in $\mathcal{T}$,
 and  any irreducible polynomial $f(x) \in K[x]$ with $f(0)=-1$.   
 
 In an arbitrary graph $E$, any sink $u\in E$ gives rise to the Chen simple module $L_K(E)u$.     Although in many contexts this type of Chen simple yields the ``easy case" (because, for example, these are always projective), the task of finding the injective envelope of such Chen simples turns out require a significantly different  approach than that used for the Chen simples arising from maximal cycles. 
We are able to explicitly construct   the injective envelope of the Chen simple module  $L_K(\mathcal{T})w$  where $w$ is the unique sink in $\mathcal{T}$; see Corollary \ref{YisenvelopeofRw}.       The tools we employ  to build the injective envelope of $L_K(\mathcal{T})w$  seem to be new within the  Leavitt path algebra literature.   

Consequently, the construction of the injective envelopes corresponding to the two types of Chen simple modules allows us to present  an explicit description of  a minimal injective cogenerator for $L_K(\mathcal{T})$-Mod (Theorem  \ref{Yismininjcogen}).    
This is the the first time in the literature that an injective cogenerator for a non-Noetherian  Leavitt path algebra is completely described. In particular, the structure of all injective $L_K(\mathcal{T})$-modules, and hence of all representations of $L_K(\mathcal{T})$, is revealed.

Because it plays a central role in our investigations, we give the following explicit description of the Leavitt path algebra $L_K(\mathcal{T})$.  For the directed graph 
$$ \mathcal{T}:=  \hskip-.1in  \xymatrix{ {}^c \hskip-.3in & {\bullet}^v \ar@(ul,dl) \ar[r]^d &{\bullet}^w},$$
\noindent
  we consider the {\it extended graph} $\widehat{\mathcal{T}}$ of $\mathcal{T}$, pictured here:
$$\widehat{\mathcal{T}}  :=  \hskip-.1in  \xymatrix{  \hskip-.3in & {\bullet}^v \ar@(l,d)_c \ar@(u,l)_{c^*} 
 \ar@/^-1.0pc/[r]_d 
&   \ar@/^-1.0pc/[l]_{d^*}
  {\bullet}^w} \  .$$
Then $L_K(\mathcal{T})$ is defined to be the standard path algebra $K\widehat{\mathcal{T}}$ of $\widehat{\mathcal{T}}$ with coefficients in $K$, modulo these relations:
$$ c^*c = v; \ \ d^*d = w; \ \ c^*d = d^*c = 0;   \ \ \mbox{and} \ \ cc^* + dd^* = v.$$
\noindent
In particular, $v+w = 1_{L_K(\mathcal{T})}$.

\begin{proposition}\label{JacobsonisLeavitt}\cite[Proposition 1.3.7]{TheBook}
Let $K$ be any field, and let $\mathcal{T} $ be the graph 
$ \hspace{-.1in}
  \xymatrix{ {}^c \hskip-.3in & {\bullet}^v \ar@(ul,dl) \ar[r]^d &{\bullet}^w}.$   Then $\R\cong L_K(\mathcal{T})$ as $K$-algebras.  
\end{proposition}
\begin{proof}
In $L_K(\mathcal{T})$ we have 
$$(c^* + d^*)(c + d) = v + 0 + 0 + w = 1_{L_K(\mathcal{T})}, \ \ \mbox{and} $$
$$  (c+d)(c^* + d^*) = cc^* + 0 + 0 + dd^* = v \neq 1_{L_K(\mathcal{T})}.$$ With this as context, one can show that the map 
$$\varphi:  \R \to L_K(\mathcal{T}) \ \ \mbox{ given by the extension of} \ \  \varphi(X) = c^* + d^*,  \ \varphi(Y)= c+d$$
is an isomorphism of $K$-algebras.     
\end{proof}

{In particular, we note for later use that the element $c$ of $L_K(\mathcal{T})$ corresponds to the element $y^2x$ of $\mathcal{R}$ under this isomorphism.} 

With Proposition \ref{JacobsonisLeavitt} in mind, for the remainder of this article we investigate the structure of  the Jacobson algebra $\R := K\langle X, Y \ | \ XY = 1\rangle$ by equivalently investigating the structure of the Leavitt path algebra $ L_K(\mathcal{T})$.

\begin{corollary}\label{quotientsofinjectives}
The Jacobson algebra is (left and right) hereditary.  Specifically, quotients of injective left $\R$-modules are injective.  
\end{corollary}

\begin{proof}
By \cite[Theorem 3.2.5]{TheBook}, the Leavitt path algebra $L_K(E)$ for any finite graph $E$ is hereditary.   That hereditary rings have (indeed, are characterized by) the specific  property follows by \cite[Theorem 3.22]{Lam}.  
\end{proof}

 An application of \cite[Corollary 1.5.12]{TheBook} gives the following useful description of a $K$-basis of $L_K(\mathcal{T})$. 
\begin{lemma}\label{basisofToeplitz}
The following set forms  a $K$-basis of  $L_K(\mathcal{T})$:  
$$v, \ w, \ d,  \ d^*, \ c^i, \ c^id, \ c^i (c^*)^j,  \ (c^*)^j, \ d^*(c^*)^j$$
where $i, j \geq 1$.  
\end{lemma}

\bigskip

\textbf{NOTATION.} Throughout, we often denote $L_K(\mathcal{T})$ simply by $R$.  We denote by $\mathbb N$ the set of positive integers $\{1,2,3,...\}$, and by $\mathbb Z^+$ the set $\mathbb N\cup\{0\} = \{0,1,2,...\}$.   

\medskip

The word ``module" will always mean ``left module".    For $f(x) \in K[x]$ and $n\in \N$  we denote $(f(x))^n$ by $f^n(x)$.  

\medskip

 For any polynomial $g(x)=\sum_{i=0}^m k_{i}x^{i}\in K[x]$, and the cycle $c$ in $\mathcal T$, we denote by $g(c)$ the element 
\[g(c):=k_01_R+k_1c+\cdots+k_mc^m\in R. \]
Rewritten, $g(c) = k_0v + k_0w +k_1c+\cdots+k_mc^m\in R$.  This notation is well-suited for our purposes, {\it but we note that this definition of $g(c)$  is different than that used for expressions of the form $g(c)$ elsewhere in the literature}. 
In certain instances we will need to clearly distinguish between our notation and these other notations.  For $g(x) = \sum_{i=0}^m k_{i}x^{i}\in K[x]$ we denote by $g\big|_v(c)$  the element
$$g \big|_v (c) := k_0v +k_1c+\cdots+k_mc^m\in R.$$
So $g(c)=k_0w+g \big|_v (c)$ and $g \big|_v (c)=vg(c)$.

 \medskip
 
 We denote by $\mathcal{F}$ the set of polynomials 
 $$\mathcal{F} := \{f(x) \in K[x] \ | \ f \mbox{ is irreducible in } K[x], \mbox{and} \ f(0)=-1_K\},$$ 
and by $\mathcal{P}$ the set of polynomials 
 $$\mathcal{P} := \{p(x) \in K[x] \ | \ p(0)\not=0\}.$$
 We note that the family $\mathcal{F}$ is a set of pairwise  non-associate representatives of the irreducible elements in the ring of Laurent polynomials $K[x,x^{-1}]$.

\medskip

\section{The simple modules over $L_K(\mathcal{T})$}\label{simplessection}

In this section we present properties of the simple left $L_K(\mathcal{T})$-modules of the form $V^f$, where $f(x)\in \mathcal{F}$.   Many of these properties were established in \cite{AR}; we will develop here some additional information about these simple modules which will be needed in the sequel.     Although we will not actually utilize the following piece of information until the final section of the article, we reiterate 
here that because there is a unique cycle in $\mathcal{T}$,  \cite[Corollary 4.6]{AR} applies, and  yields that, up to isomorphism, all but one of the simple modules over $L_K(\mathcal{T})$  are of the indicated form.   (The only other simple $L_K(\mathcal{T})$-module is $L_K(\mathcal{T})w$.)

We now make a detailed presentation of the  construction of the modules $V^f$.
%
%
Assume $f(x)\in \mathcal{F}$ has degree $n$. Denote by $K'$ the field $K[x]/\langle f(x)\rangle$ and by $\overline x$ the element $x+\langle f(x)\rangle$. Clearly $\{1, \overline x, ..., \overline x^{n-1}\}$ is a $K$-basis of $K'$. Consider the  infinite path $c^{\infty}$; the class of infinite paths tail equivalent to $c^\infty$ consists only of $c^\infty$ itself. 
Let $V^{\overline x}$ be the one dimensional $K'$-vector space generated by $c^\infty$. Setting
\[d\star c^\infty=d^*\star c^\infty=w\star c^\infty = 0;\]
\[ \ v\star c^\infty=c^\infty;  \ c\star c^\infty=\overline x c^\infty; \ \mbox{and}   \ c^*\star c^\infty={\overline x}^{-1}c^\infty,\]
$V^{\overline x}$ becomes a left $L_{K'}(\mathcal{T})$-module.
Consider the linear maps
\[\sigma^{\overline x}:V^{\overline x}\to K', \quad h\cdot c^\infty\mapsto h,\text{ and}\]
\[\rho^{\overline x}:K'\to V^{\overline x}, \quad h\mapsto h\cdot c^\infty.\]
Clearly these maps are inverse isomorphisms of one-dimensional $K'$-vector spaces.

\medskip

\noindent Restricting the scalars to $K$, the abelian group $V^{\overline x}$ has also a  left $R$-module structure: we denote this left $R$-module by $V^{f}$. A $K$-basis of $V^{f}$ is $\{c^\infty, \overline x c^\infty, ..., {\overline x}^{ n-1}c^\infty\}$. 
Denote by $G^f$ the $K$-subspace of $R$ generated by $\{1, c, ..., c^{n-1}\}$. (We note for later that clearly any element in $G^f$ commutes with $f(c)$.). 
The linear maps
\[\sigma^f:V^f\to G^f , \quad {\overline x}^{ i}c^\infty\mapsto c^i,\text{ and}\]
\[\rho^f:G^f \to V^f, \quad c^i\mapsto {\overline x}^{ i}c^\infty\ \]
(for $0\leq i \leq n-1$) define inverse isomorphisms of $n$-dimensional $K$-vector spaces. The map $\rho^f$ is the restriction to $G^f$ of the right multiplication map by $c^\infty$:
\[\rho:R\to V^f,\ r\mapsto r\star c^\infty.\]
Clearly one has
\[\sigma^f({\overline x}^{ i}c^\infty)\star c^\infty=c^i\star c^\infty={\overline x}^{ i}c^\infty=\rho^f(c^i)=\rho(c^i).\]

\begin{lemma}\cite[Lemma 3.3]{AR} \label{Lemma:Vsemplice}. Let $f(x) \in \mathcal{F}$.  Then the left $R$-module $V^f$ is simple.
\end{lemma}
\begin{proof}
Let $U$ be a nonzero $R$-submodule of $V^f$. Since $\{1,\overline x, ..., {\overline x}^{ n-1}\}$ is a $K$ basis for $K'$ and
\[\overline x\cdot u=c\star u\ \ \ \  \forall u\in U,\]
$U$ is also a $K'$-space. Since $V^{\overline x}$ is a one-dimensional $K'$-space, we have $U=V^{\overline x}$ as $K'$-spaces and hence $U=V^f$ as left $R$-modules.
\end{proof}

\medskip

\noindent The next result shows that the $V^f$'s are finitely presented, and determines their annihilators.
\begin{lemma}\label{lemma:seb}
Let $f(x) \in \mathcal{F}$. Denoting by $\widehat{\rho}_{f(c)}:R\to R$ the right multiplication map by $f(c)$, we have the following short exact sequence of left $R$-modules:
\[\xymatrix{0\ar[r]&R\ar[r]^{\widehat{\rho}_{f(c)}}&R\ar[r]^\rho&V^f\ar[r]&0}.\]
In particular:
\begin{enumerate}
\item The kernel of $\rho:R\to V^f$ is $Rf(c)$.
\item $Rf(c)$ coincides with the two-sided ideal $\Ann_R(V^f)$.
\end{enumerate}
\end{lemma}
\begin{proof}
We have already observed that $\rho^f$ is surjective, and thus  $\rho$ is as well.

For the injectivity of $\widehat{\rho}_{f(c)}$, we note that any element $f(x)\in\mathcal F$ can be written as $f(x)=xg(x)-1$, and so $f(c)=cg(c)-1\in R$,
 for a suitable polynomial $g(x)\in K[x]$.  Let $r\in R$ such that $\widehat{\rho}_{f(c)}(r)=0$.   So $r(cg(c)-1) = 0$, thus  $rcg(c) = r$, which recursively gives $r(cg(c))^j = r$ for any $j \geq 1$.  Now write $r=\sum_{i=1}^{n} k_i{\alpha_i}{\beta_i}^*$, where the $\alpha_i$ and  $\beta_i$ are in $ {\rm Path}(\mathcal{T})$.   We note that, for any   $\beta \in {\rm Path}(\mathcal{T})$, there exists a suitable $m_{\beta}$ such that    $\beta^* (cg(c))^{m_\beta}$ is either $0$ or an element of $K\mathcal{T}$.   
Now let $N$ be the maximum  in the set $\{ m_{\beta_1}, m_{\beta_2}, \dots, m_{\beta_n}\}$.  Then the above discussion shows that   $r (cg(c))^N$ is an element of $R$ of the form $\sum_{i=1}^{n} k_i \gamma_i$, where  $\gamma_i \in K\mathcal{T}$ for $1\leq i \leq n$; that is, $r (cg(c))^N \in K\mathcal{T}$.     But $r (cg(c))^N = r$, so that $r \in K\mathcal{T}$.    However,  the equation  $r(cg(c)) = r$  (i.e., $rf(c) = 0$)    has only the zero  solution in $K\mathcal{T}$   by a degree argument.      So $r = 0$.

We now show $\Ker \rho=Rf(c)$.  Using \cite[Lemma 3.2]{AR},  we get that 
 the annihilator of $V^f$ is the two sided ideal $I=\langle w, f\big|_v(c)\rangle$. Notice that $w=-w(cg(c)-1)=-wf(c)$;
therefore, in the notation used herein, we have  $I=\langle f(c)\rangle$. Clearly $I$ is contained in the kernel of $\rho$.
Let $r\in \Ker\rho$; to prove that $r\in I$ we have to check that $r\star \overline x^{i}c^\infty=0$ for $i=0,...,  n-1$, in other words, that  left multiplication by $r$ annihilates all the elements of a $K$-basis of $V^f$. We consider the left $L_{K'}(\mathcal{T})$-module $V^{\overline x}$. Since $\overline x^{ i}$ is a scalar in  $L_{K'}(\mathcal{T})$, and $r\star c^\infty=0$ we have the following equality in $V^{\overline x}$:
\[r\star \overline x^{i}c^\infty=\overline x^{ i} r\star c^\infty=0.\]
Since $V^{\overline x}=V^f$ as abelian groups, the desired result follows.  
%
%
%
%
%
\end{proof}

\begin{remark}\label{injectiveusesournotation}  The injectivity of right multiplication by $f(c)$ relies on the fact that the
constant term of $f(c)$ is $- 1_R$.
\end{remark}

\begin{lemma}\label{lemma:intersezione}
For any $f(x)\in \mathcal{F}$, the intersection of $Rf(c)$ with $G^f$ is 0.
\end{lemma}
\begin{proof}
If $\ell
$ belongs to $Rf(c)\cap G^f$, then $\rho(\ell)=0$ by Lemma \ref{lemma:seb}(1), so that 
$$0=\sigma^f(0) = \sigma^f(\rho(\ell))=\sigma^f(\rho^f(\ell)) =\ell$$
 (using $\rho^f(\ell) = \rho(\ell)$ since $\ell \in G^f$). 
\end{proof}

\begin{proposition}[The Division Algorithm by $f(c)$]\label{prop:divalg}   Let $f(x) \in \mathcal{F}$.  For any $\beta\in R$ there exists unique $q_\beta\in R$ and $r_\beta\in G^f$ such that
\[\beta=q_\beta f(c)+r_\beta.\]
\end{proposition}

\begin{proof}
Consider the element $r_\beta:=\sigma^f(\rho(\beta))$. Clearly $r_\beta$ belongs to $G^f\subseteq R$. Let us prove that the difference $\beta-r_\beta$ belongs to $\Ker\rho$. By Lemma~\ref{basisofToeplitz}, it is sufficient to prove that $\beta-r_\beta$ belongs to $\Ker\rho$ for $\beta\in\{v, w, d, c^i, c^id, c^i (c^*)^j,  (c^*)^j, d^*(c^*)^j\}$. Whenever $\rho(\beta)=0$, then also $r_\beta=0$ and hence $\beta-r_\beta$ belongs to $\Ker\rho$ in these cases; so the result immediately holds for $\beta = w, d,  c^id,$ and $  d^*(c^*)^j$. For the others:
\[r_v=\sigma^f(c^\infty)=
1_K,\ r_{c^i}=\sigma^f(\overline x^i c^\infty)=c^i,\ r_{c^i(c^*)^j}=\sigma^f(\overline x^{i-j} c^\infty)=\begin{cases}
   c^{i-j}   & \text{if }i>j\geq 0, \\
    1_K   & \text{if }i=j\geq 0, \\
     (c^*)^{j-i}   & \text{if }0\leq i<j, 
\end{cases}\]
and clearly
$\ v-1_K$, $c^i-r_{c^i}$ (which is $0$),  and $c^i(c^*)^j-c^{i-j}$ for $i>j$, $c^i(c^*)^i-1_K$,  $c^i(c^*)^j-(c^*)^{j-i}$ for $i<j$, belong to $\Ker\rho$. By Lemma~\ref{lemma:seb},   $\Ker\rho=Rf(c)$; therefore $\beta-r_\beta=q_\beta f(c)$ for a suitable $q_\beta\in R$. 

We now prove that $q_\beta\in R$ and $r_\beta\in G^f$ are uniquely determined.
Assume
\[\beta=q_1f(c)+r_1=q_2f(c)+r_2;\]
then we have $r_1-r_2=(q_2-q_1)f(c)\in Rf(c)\cap G^f$, which is 0 by Lemma~\ref{lemma:intersezione}.
Therefore $r_1=r_2$ and $(q_1-q_2)f(c)=r_1-r_2=0$; since by Lemma~\ref{lemma:seb} right multiplication by $f(c)$ is injective, we have $q_1=q_2$.
\end{proof}

\section{The Pr\"ufer modules $U^f$}

For any simple module $V^f$ there exists a uniserial module of infinite length, with all the composition factors isomorphic to $V^f$. We call this module \emph{the Pr\"ufer module} associated to $V^f$, and its construction comes from the general setting described in \cite{AMTPrufer}.

\begin{lemma}\label{lemma:notzerodivnotleftinv}
For any $f(x)\in\mathcal F$, the element $f(c)\in R$ is neither a right zero divisor nor left-invertible.
\end{lemma}
\begin{proof}
The element $f(c)\in R$ is not a right zero divisor, since the right multiplication $\widehat{\rho}_{f(c)}:R\to R$ is injective by Lemma~\ref{lemma:seb}.   By that same Lemma we also have   $f(c)\star c^\infty=0$ in $V^f$,  and so $f(c)$ is not left invertible in $R$.
\end{proof}

The upshot of Lemma \ref{lemma:notzerodivnotleftinv} is that  we can apply the construction of the Pr\"ufer module described in \cite[Section 2]{AMTPrufer} with $a=f(c)$. For each natural number $n\geq 1$, set
\begin{itemize}
\item $M_n^f:=R/Rf^n(c)$, the non-zero cyclic left $R$-module generated by $1 + Rf^n(c)$.
\item $\eta^f_n:R\to M_n^f$ the canonical projection.
\item $\theta^f_n:Rf(c)\to M^f_n$, $f(c)\mapsto 1+Rf^n(c)$.
\item $\psi_{i,\ell}:M_i^f\to M_\ell^f$, $1+Rf^i(c)\mapsto f^{\ell-i}(c)+Rf^\ell(c)$ for each $i\leq\ell$; the cokernel of $\psi_{i,\ell}$ is isomorphic to $M_{\ell-i}^f$.
\end{itemize}
With this notation, the diagram
$$\xymatrix{ R \ar[r]^{\widehat{\rho}_{f(c)}}\ar[d]_{{\eta^f_{n-1}}} & R \ar[d]^{{\eta^f_{n}}}\\
 M_{n-1}^f\ar[r]^{\psi_{n-1,n}} & M_n^f  } \ \ \mbox{or,  equivalently,  the diagram }  
  \xymatrix{ Rf(c) \ar@{^(->}[r]^{i}\ar[d]_{{\theta^f_{n-1}}} & R \ar[d]^{{\eta^f_{n}}}\\
  M_{n-1}^f\ar[r]^{\psi_{n-1,n}} & M_n^f }
$$
 \noindent
   is a pushout diagram.
By Lemma~\ref{Lemma:Vsemplice}, $M_1^f\cong V^f$ is a simple $R$-module. 

We now establish the key property of the modules $\{M_i^f \ | \ i\in \N\}$ which will allow us to further apply additional machinery built in \cite{AMTPrufer}.  
\begin{lemma}\label{lemma:equation}
Let $f(x) \in \mathcal{F}$.  Then the equation $f(c)\mathbb X=1+Rf^{n}(c)$ has no solutions in the left $R$-module $M^f_{n}$.
\end{lemma}
\begin{proof}
Let $m+Rf^n(c)\in M^f_n$, with $m\in R$. By a repeated application of Proposition~\ref{prop:divalg}, we have
$$m=q_1f(c)+g_1, \ q_1=q_2f(c)+g_2, \ ... \ ,  \ q_{n-1}=q_{n}f(c)+g_{n},$$
 where the elements $g_i$ ($1\leq i \leq n)$ belong to $G^f$. Therefore
 \[m-\big(g_1+g_2f(c)+\cdots+g _nf^{n-1}(c)\big)\in Rf^n(c).\]
 In particular we can assume that the representative $m$ of the coset $m+Rf^n(c)$ is equal to $g_1+g_2f(c)+\cdots+g _nf^{n-1}(c)$.
Assume $f(c)m+Rf^n(c)=1+Rf^n(c)$. Then $f(c)m-1$ belongs to $Rf^n(c)$.
Therefore 
\[f(c)\big(g_1+g_2f(c)+\cdots+g _nf^{n-1}(c)\big)-1\]
belongs to $Rf^n(c)$. Since as noted above $f(c)g_i=g_if(c)$ for each $1\leq i \leq n$, we get
\[-1+g_1f(c)+ g_2f^2(c) + ...+g_{n}f^{n}(c)\in Rf^n(c).\]
Then $-1=rf(c)$ for a suitable $r\in R$ and hence $f(c)$ would be left invertible in $R$, which contradicts Lemma \ref{lemma:notzerodivnotleftinv}. 
\end{proof}

With Lemma \ref{lemma:equation} established, we may apply \cite[Proposition~2.2]{AMTPrufer} to conclude that each left $R$-module $M^f_n$ ($n\in \N$)  is uniserial of length $n$. We define
 $$U^f:=\varinjlim\{M_i^f,\psi_{i,j}\}_{i\leq j},$$
  and, for each $i\in \N$, the induced monomorphism
  $$\psi_i:M_i^f\to U^f.$$
By  \cite[Proposition~2.4]{AMTPrufer}, 
 $U^f$ is uniserial and artinian.

For each $n\in \mathbb N$, the element 
$$\alpha_{n,f} :=\psi_n(1+Rf^n(c))$$
 is a generator of the submodule $\psi_n(M_n^f)$ of $U^f$. In the sequel, to simplify the notation,  we will denote by $M_n^f$ the submodule $\psi_n(M_n^f)$ of $U^f$, in fact identifying $M_n^f$ with its image in $U^f$ through the monomorphism $\psi_n$.  Let $r\alpha_{n,f}=r+Rf^n(c)$ be a generic element of $U^f$; applying  the Division Algorithm (Proposition \ref{prop:divalg})  $n-1$ times, we get
\begin{eqnarray*}
r\alpha_{n,f}=r+Rf^n(c)& =& g^f_0+g^f_1f(c)+\cdots+g^f_{n-1}f^{n-1}(c)+Rf^n(c) \\
&   =&
\big(g^f_0+g^f_1f(c)+\cdots+g^f_{n-1}f^{n-1}(c)\big)\alpha_{n,f}
\end{eqnarray*}
for suitable $g^f_0$, ..., $g^f_{n-1}\in G^f$.

\medskip

\begin{remark}\label{Vfnotinjective} The goal of this article is to study injective modules over $L_K(\mathcal{T})$.  As an immediate consequence of Lemma \ref{lemma:equation}, we see that any $L_K(\mathcal{T})$-module of the form $M_i^f$ is {\it not} injective, because the map $\psi: R \to M_i^f$ defined by setting $\psi(1) = 1+ Rf^i(c)$ does not factor through the monomorphism $\widehat{\rho}_{f(c)}: R\to R$.  In particular, the simple module $M_1^f \cong V^f$ is not injective.    However,  in the next section,   we will show that each $U^f =\varinjlim\{M_i^f\} $ {\it is}  an injective left $L_K(\mathcal{T})$-module.
\end{remark}

\section{The left ideals in $L_K(\mathcal{T})$}

In order to test whether a module is injective by applying Baer's criterium, 
we must have available a complete description of the left  ideals in $L_K(\mathcal{T})$. We will show that any ideal of $R$ is either a direct summand of a left ideal of the form $Rp(c)$ (where $p(x) \in K[x]$ has $p(0) = 1$), or a direct summand of ${\rm Soc}(R)$. We recall that $\mathcal P$ denotes the set of polynomials $p(x)\in K[x]$ with $p(0)\not=0$.

\begin{remark}\label{radicalproperties}
We collect up in this remark some properties of $J := {\rm Soc}(L_K(\mathcal{T}))$,  the socle of $R$.   It is well-known (or see \cite[Theorem 2.6.14]{TheBook}) that $J = \langle w \rangle$ as a two-sided ideal. Further, as left $R$-ideals, $$J = Rw \oplus (\oplus_{i\in \mathbb Z^+} Rc^idd^* (c^*)^i) = Rw \oplus (\oplus_{i\in \mathbb Z^+} Rd^* (c^*)^i) .$$ 
Moreover, each summand of the form  $Rc^idd^* (c^*)^i$ is isomorphic to the simple module $Rw$.

  It has been noted many places (see e.g.  \cite[Example 4.5]{Ranga}) 
   that  $R / J \cong K[x,x^{-1}]$ as $K$-algebras.   This isomorphism is also as left $R$-modules (and left $R/J$-modules), which is not hard to see directly.  Indeed, the  standard monomials in  $R$  end (on the right)  in a term having one of the forms $v, w, d, d^*, c^i, c^id, c^i (c^*)^j,  (c^*)^j,$ or $ d^*(c^*)^j$; and we have
\[w\equiv d\equiv d^* \equiv c^id\equiv d^*(c^*)^j\equiv 0 \ \mbox{mod} J, \ \ \mbox{while} \ \ 
v\equiv  cc^*\equiv c^*c \equiv 1  \ \mbox{mod} J.\]
So the only terms which survive mod $J$ are powers of $c$ (positive or negative).

Then the standard bijective correspondence between left ideals of $R$ which contain $J$ and submodules of $R/J$, together with the well-known principal ideal structure of $K[x,x^{-1}]$, yields that every left ideal of $R$ which properly contains $J$ is of the form $Rp\big|_v(c)$ for some $p(x)\in \mathcal P$. 
But because $w\in J$, and we are assuming that $J \subseteq Rp\big|_v(c) $, we get that $Rp\big|_v(c) = Rp(c)$.   The upshot is that every left ideal of $R$ which properly contains $J$ is of the form $Rp(c)$ for some $p(x) \in \mathcal P$.  \hfill $\Box$ 
\end{remark}

\medskip

\begin{proposition}\label{homJVfis0}
Let $f(x)\in \mathcal{F}$.  Then ${\rm Hom}_R(J, U^f) = \{0\}$. 
\end{proposition}
\begin{proof}
For any $f(x)\in \mathcal{F}$ we have ${\rm Hom}_R(Rw, V^f) \cong w V^f = \{0\}$, because $V^f $ is generated as a $K$-space by elements of the form $\overline{x}^i c^\infty$ ($0\leq i \leq {\rm deg}(f)-1$), and $w \overline{x}^i c^\infty=  \overline{x}^i wc^\infty= 0$.  

By \cite[Proposition 2.2]{AMTPrufer}, the  composition factors of the finitely generated submodules of $U^f$ are isomorphic to $V^f$.   This together with the previous paragraph gives  ${\rm Hom}_R(Rw, U^f) = \{0\}$.

As noted in Remark \ref{radicalproperties},  $J = Rw \oplus (\oplus_{i\in \mathbb Z^+} Rc^idd^* (c^*)^i) \cong \oplus_{i\in \mathbb Z^+}Rw.$   Then ${\rm Hom}_R(J, U^f) \cong {\rm Hom}_R(\oplus_{i\in \mathbb Z^+}Rw, U^f) \cong \prod_{i\in \mathbb Z^+}{\rm Hom}_R(Rw, U^f) =  \prod_{i\in \mathbb Z^+}\{0\} = \{0\}$.
\end{proof}

\begin{proposition}\label{structureofleftideals}
Let $I$ be a left ideal of $R$.   Then either:

1) There exists $p(x) \in \mathcal{P}$ for which $I$ is a direct summand of $Rp(c)$, or 

2)  $I$ is a direct summand of $J = {\rm Soc}(R)$.

\end{proposition}

\begin{proof}
Case 1.     $J$ is properly contained in $I$.   Then as noted in Remark \ref{radicalproperties} we have $I = Rp(c)$ for some $p(x) \in \mathcal P$, and so we are done in this case.

\medskip

Case 2.   Suppose $I$ is not contained in $J$, and $I$ does not contain $J$.   Consider the left ideal $A = I+J$.  Then $A$ properly contains $J$, so we may apply the Case 1 analysis to $A$, so that $A = Rq(c) $ for some $q(x) \in\mathcal P$.  
Since the socle $J$ is a direct sum of simple left $R$-modules, we have $J = (I\cap J) \oplus B$ for some left ideal $B$ of $R$ contained in $J$.   It is straightforward to show that this gives $A = I \oplus B$.     But then $I$ has been shown to be a direct summand of $A = Rq(c)$, as desired.  
\medskip

Case 3.   Suppose $I$ is contained in $J$.  Then the semisimplicity of $J$ immediately gives that $I$ is a direct summand of $J$.   
\end{proof}

{
\begin{remark}\label{GerritzenandIovanovknewtheideals}
We note that Gerritzen in \cite[Proposition 3.4]{Gerritzen} established that all one-sided ideals of the Jacobson algebra $\mathcal{R}$ are either principal, or contained in the socle of $\mathcal{R}$.     Similarly, Iovanov and Sistko in \cite[Theorem 2 and Corollary 1]{IovanovSistko} establish the same type of result in $\mathcal{R}$, in terms of polynomials in the element $x$ of $\mathcal{R}$.  By a previous observation, the element $c$ of $L_K(\mathcal{T})$ corresponds to the element $y^2x$ of $\mathcal{R}$.   The point to be made here is that while these two results from \cite{Gerritzen} and \cite{IovanovSistko}  are clearly related to the conclusion of  Proposition \ref{structureofleftideals},    Proposition \ref{structureofleftideals} gives a more explicit description of these left ideals,  in a form which will be quite useful for us in the sequel. 
\end{remark}
}

\begin{corollary}\label{simplifyBaer}
In order to apply the Baer criterion to determine the injectivity of a left $R$-module, we need only check injectivity with respect to $J$, and with respect to left ideals of the form $Rp(c)$ for $p(x) \in \mathcal P$. 
\end{corollary}

\section{A (minimal) injective cogenerator for $L_K(\mathcal{T})$ }

In this final section we use the machinery developed above to achieve the main goal of this article, namely, to identify a minimal injective cogenerator for $L_K(\mathcal{T})$.   In the first portion of the section we show that the injective envelope of each of the simple modules $V^f$ is the Pr\"{u}fer module $U^f$.  We then proceed to construct, using completely different methods, the injective envelope of the simple module $Rw$.   We finish the section by appropriately combining these two types of injective modules.  

In previous work by the three authors \cite{AMTPrufer},  { modules of the form $U^{x-1}$ over general Leavitt path algebras $L_K(E)$  were shown to be injective, in case the corresponding cycle $c$ is maximal.   
 Establishing injectivity   of such $U^{x-1}$ over the Leavitt path algebra $L_K(E)$ of an arbitrary finite graph $E$ } required an analysis of the structure of  $U^{x-1}$ viewed as a right module over its endomorphism ring.   In the present setting, we need not invoke this right module structure, the reason being that in the particular case $R = L_K(\mathcal{T})$ we have a complete description of the left $R$-ideals, and therefore we are in position to productively use Baer's criterion to establish injectivity of left $R$-modules.  

\bigskip

\subsection{The injective envelope of $V^f$}\label{subsectionV^f}

We start by establishing that  $U^f$ is injective for any $f(x)\in \mathcal{F}$.   By Proposition \ref{homJVfis0} we have $\Hom_R(J,U^f) = 0$.  So   by Corollary~\ref{simplifyBaer}, in order to check the injectivity of $U^f$ it is enough to check the Baer criterion with respect to left ideals of the form $Rp(c)$ for $p(x) \in \mathcal P$.

\begin{lemma}\label{lemma:coprime}
Let $f(x)\in \mathcal{F}$, and let $g(x) \in K[x]$ which is not divisible by $f(x)$. Then there exists a polynomial $\beta(x)\in K[x]$ such that $\beta(c)g(c) \in 1+Rf^n(c)$. In particular,  $g(c)+Rf^n(c)$ is a generator of the uniserial module $M_n^f$.
\end{lemma}
\begin{proof}
Since $f(x)$ is irreducible, non-divisibility gives $\text{gcd}(f^n(x), g(x))=1$. Then there exist polynomials $\alpha(x)$, $\beta(x)\in K[x]$ such that $1=\alpha(x)f^n(x)+\beta(x)g(x)$. Therefore
$\beta(c)g(c)=1-\alpha(c)f^n(c)$ and hence $\beta(c)g(c)\in 1+Rf^n(c)$.
\end{proof}

\begin{proposition}
Let $f(x) \in \mathcal{F}$.  Then the uniserial left $R$-module $U^f$ is injective.
\end{proposition}
\begin{proof}
By Proposition \ref{homJVfis0} and Corollary \ref{simplifyBaer}, it suffices to show, for any  $p(x)\in \mathcal P$ and  $\varphi:Rp(c)\to U^f$, that $\varphi$ extends to a map $\overline \varphi:R\to U^f$.  Clearly the zero  map extends to $R$.  So suppose  $\varphi\not=0$. Let $n\in\mathbb N$ be minimal such that $\Im\varphi \subseteq M_{n}^f$, and  write $\varphi(p(c))=m+Rf^n(c)$ for some $m\in R$.  As noted in the proof of Lemma \ref{lemma:equation}, we can choose
\[m=g_1+g_2f(c)+\cdots+g _nf^{n-1}(c),\]
where $g_i \in G^f$ ($1\leq i \leq n$).   In particular,  
 $m$ commutes with all polynomials in $c$.

By the construction of the direct limit $U^f$,  for each $i\geq 0$ we have
\[\varphi(p(c)) =  m + Rf^n(c) =mf^i(c)+Rf^{n+i}(c)=f^i(c)m+Rf^{n+i}(c).
\]
Let $p(x)=f^\ell(x)p_0(x)$ with $\ell\geq 0$ and $f(x)  \nmid  p_0(x)$. By Lemma~\ref{lemma:coprime} there exists $\beta_0(x)\in K[x]$ such that $\beta_0(c)p_0(c)=p_0(c)\beta_0(c)$ belongs to $1+Rf^{n+\ell}(c)$. 
Therefore
\[p(c)(\beta_0(c)m+Rf^{n+\ell}(c))=f^\ell(c)p_0(c)\beta_0(c)m+Rf^{n+\ell}(c)=
f^\ell(c)m+Rf^{n+\ell}(c)=\varphi(p(c)).\]
Thus the morphism $\overline{\varphi}: R\to U^f$ defined by setting $\overline{\varphi}(1) = \beta_0(c)m+Rf^{n+\ell}(c)$ extends $\varphi$.
\end{proof}

\begin{corollary}\label{UfisenvelopeofVf}
Let $f(x) \in \mathcal{F}$.  Then $U^f$ is the injective envelope of $V^f$.
\end{corollary}
\begin{proof}
The simple module $V^f$ is essential in $U^f$, since $U^f$ is uniserial.  But the injective envelope of any module is an injective module in which the given module sits as an essential submodule.   
\end{proof}

Of course in general the direct sum of infinitely many injective modules need not be injective.   (Over an arbitrary ring $S$, any infinite direct sum of injectives is injective if and only if $S$ is Noetherian; and clearly $R = L_K(\mathcal{T})$ is non-Noetherian, because, for example, $J$ is a non-finitely-generated left ideal of $R$.)    This observation notwithstanding, we close this subsection with the following.  
\begin{proposition}\label{prop:U}
  Let $U=\oplus_{\lambda\in \Lambda} I_{\lambda}$ where, for each $\lambda\in \Lambda$, there exists  $f(x) \in\mathcal{F}$ such that $I_{\lambda}$ is an injective module isomorphic to $U^f$. Then $U$ is injective.
\end{proposition}
\begin{proof}
We again invoke   Corollary \ref{simplifyBaer}.  So we need only establish two steps.  

Step 1: Consider the ideal $J$ and let $\varphi:J\to U$. We show that $\varphi=0$.   Suppose otherwise.   The  image of  $\varphi$ is a  semisimple module, isomorphic to a direct sum of copies of $Rw$. But each $U^f$ has essential socle isomorphic to $V^f$ and so the socle of $U$ is isomorphic to  the direct sum of copies of the $V^f$'s. Since $Rw \not\cong V^f$ for any $f$,  we get a contradiction.

Step 2:   Let $p(x)$ be a polynomial in $\mathcal P$.   If $\varphi: Rp(c) \rightarrow U$ then the image of $\varphi$ is  finitely generated, and so is contained in $\hat{U} \cong  \oplus_{i=1}^n U^{f_i}$ for some appropriate $f_i$'s.  But $\hat{U}$ is injective because each $U^{f_i}$ is (and the sum is finite), and so $f$ extends. 
\end{proof}

\bigskip

\subsection{The injective envelope of $Rw$}\label{subsectionRw}

Having identified the injective envelope of each of the simple modules $V^f$ ($f(x) \in \mathcal{F}$), we now turn our attention to identifying the injective envelope of the simple module $Rw$.   

\begin{lemma}\label{basisofRw}
The set $\{w,d,cd,c^2d,..., c^id,...\}$ is a $K$-basis of the simple module $Rw$.  That is, any element of $Rw$ can be written uniquely as $kw+\sum_{i=0}^n k_ic^id=kw+(\sum_{i=0}^n k_ic^i)d$, with $k, k_i \in K$.
\end{lemma}  
\begin{proof}
It is easily shown that $Rw=Rd^*d = Rd$. By Lemma~\ref{basisofToeplitz}, the elements
\[v, w, d, d^*, c^i, c^id, c^i (c^*)^j,  (c^*)^j, d^*(c^*)^j\quad i, j\geq 1
\]
form a $K$-basis of $R$. Since 
\[0=wd=dd=c^idd=c^i (c^*)^jd=(c^*)^jd=d^*(c^*)^jd \quad \forall i, j\geq 1\]
we conclude that a basis of $Rw=Rd$ is formed by multiplying  the remaining elements of the $K$-basis for $R$ on the right by $d$, namely  
\[ vd=d, \ \  d^*d = w, \ \  \mbox{and} \  c^id  \ \ ( i \geq 1),
\]
which gives the result. 
\end{proof}

In the following sense, the simple module $Rw$ behaves similarly to the simple modules $V^f$ (see Remark \ref{Vfnotinjective}).

\begin{proposition}\label{prop:Rwnotinj}
The left $L_K(\mathcal{T})$-module $L_K(\mathcal{T})w$ is not injective.   In particular, the map $\chi = \rho_d: R \to Rw$ (via $1 \mapsto d$) does not factor through the monomorphism $\widehat{\rho}_{f(c)}: R \to R$ for any $f(x)\in \mathcal{F}$.  
  \end{proposition}
\begin{proof}
Write $f(c)=-1+h_1c+\cdots+h_mc^m$ with $h_m\not=0$ ($m\geq 1$).  The existence of a map $\xi:R\to Rw$ such that $\xi\circ\widehat{\rho}_{f(c)}=\chi$ is equivalent to the solvability of the equation $f(c)x=d$ in $Rw$.    We show that no such $x\in Rw$ exists.   Assume to the contrary that there is such a solution; necessarily $x\neq 0$.   By Lemma \ref{basisofRw} we may write $x = kw+\sum_{i=0}^n k_ic^id$ for some (unique) $k, k_0, \dots, k_n \in K$, where not all of these are $0$.  Then   $f(c)x=d$ gives 
\[f(c)\big(kw+\sum_{i=0}^n k_ic^id\big)=d.\]
Multiplying  both sides of this equation on the left by $w$ we get $-kw=0$, so $k=0$.   This yields that there are nonzero terms among the elements $k_0, k_1, \dots, k_n$; we may assume $k_n \neq 0$.  Now we have 
\[f(c)\big(\sum_{i=0}^n k_ic^id\big)=d.\]
  But this is impossible, as follows.  Expanding the product $f(c)\big(\sum_{i=0}^n k_ic^id\big)$, we see the coefficient on the $c^{m+n}d$ term is $h_mk_n$.  But the equation  $f(c)\big(\sum_{i=0}^n k_ic^id\big) = d$  gives that the coefficient on the $c^{m+n}d$ term is $0$.   So we get $0 = h_mk_n$  which, as $h_m\neq 0$, gives $k_n = 0$, a contradiction.  
\end{proof}

We seek to describe the injective envelope of $Rw$.  With Proposition \ref{prop:Rwnotinj} in hand, this process will require us to build a module which is strictly larger than $Rw$.
   Intuitively, we build such an injective $R$-module using an approach similar to the one in which the injective envelope of $K[x]$ (as a module over itself) is built:  to wit, the module consisting of formal power series $K[[x]]$.  

\begin{definition}
Let $Y$ denote the $K$-space whose elements are  ``formal series'' of the form
\[ Y := \{k_{-1}w+k_0d+k_1cd+\cdots+k_ic^id+\cdots \ | \ k_i \in K\}.\]
\end{definition}

The $K$-space $Y$ has a natural structure as a left $R$-module, where for $y=k_{-1}w+k_0d+k_1cd+\cdots+k_ic^id+\cdots$ one defines  
$$ c\cdot y =  k_0cd+k_1c^2d+\cdots+k_ic^{i+1}d+\cdots;\ \ \ \   c^*\cdot y=k_1d+\cdots+k_ic^{i-1}d+\cdots; $$
$$d\cdot y 
=  k_{-1}d; \ \ \ \  
d^*\cdot y=k_0w; \ \ \ \  \mbox {and} \ \ v\cdot y=k_1cd+\cdots+k_ic^id+\cdots,\ w\cdot y=k_{-1}w.$$

By Lemma \ref{basisofRw},  $Rw $ is the $R$-submodule of $Y$ consisting of those elements for which $k_i = 0$ for all $i>N$ for some $N \in \N$, i.e., $Rw$ consists of the ``standard polynomials" in $Y$.

\begin{lemma}\label{productwithy}
Let $y = k_{-1}w+k_0d+k_1cd+\cdots+k_ic^id+\cdots \in Y$. 

1) $wy = k_{-1}w$.

2)  $d^*(c^*)^j y = k_jw$ for all $j\geq 0$.

\end{lemma}

\begin{proof}  (1) is obvious, and (2) follows directly from the observation that   $d^*(c^*)^j c^id = w$ if $i=j$, and is $0$ otherwise.
\end{proof}

\begin{lemma}\label{RwessentialinY}
The simple module $Rw$ is essential in $Y$.   In particular,  $Rw = {\rm Soc}(Y)$.  
\end{lemma}
\begin{proof}
Consider an element $0\not=y=k_{-1}w+k_0d+k_1cd+\cdots+k_ic^id+\cdots\in Y$. There exists $\ell\in \mathbb{Z}^+ \cup \{-1\}$ such that $k_\ell \not=0$.   If $k_{-1} \neq 0$,  then by Lemma \ref{productwithy}(1) $wy = k_{-1}w \neq 0$ is in $ Rw$.  If $k_\ell \neq 0$ for $\ell \geq 0$,  then by Lemma \ref{productwithy}(2) $d^*(c^*)^{\ell} y=k_{\ell}w \neq 0$ 
is in $Rw$.
\end{proof}

\begin{lemma}\label{extendfromJ}
Any $R$-homomorphism from $J=Rw\oplus Rd^*\oplus Rd^*c^*\oplus Rd^*(c^*)^2\oplus\cdots$ to $Y$ extends to an $R$-homomorphism from $R$ to $Y$.
\end{lemma}
\begin{proof}
Let $\varphi:J\to Y$ be a homomorphism of left $R$-modules.    For each $i \geq 0$ let $k_i$ denote the $K$-coefficient of $w$ in the formal power series expression for $\varphi(d^*(c^*)^i)$, and let $k_{-1}$ be the $K$-coefficient of $w$ in $\varphi(w)$.   Since $\varphi(w) = \varphi(w^2) = w\varphi(w)$ and $\varphi(d^*(c^*)^i)= \varphi(wd^*(c^*)^i) = w \varphi(d^*(c^*)^i)$, Lemma \ref{productwithy} gives that $\varphi(w) = k_{-1}w$ and $\varphi(d^*(c^*)^i) = k_iw$ for all $i\geq 0$. 

Now consider the $R$-homomorphism $ \Phi:R\to Y$   obtained by  setting
\[ \Phi(1): =  k_{-1}w+k_0d+k_1cd+k_2c^2d+\cdots \ . \]
Since $\Phi(w)=w\Phi(1)=k_{-1}w$ and $\Phi(d^*(c^*)^i)=d^*(c^*)^i\Phi(1)=k_iw$ for each $i\geq 0$ (again by Lemma \ref{productwithy}),  $\Phi$ extends $\varphi$.
\end{proof}

It is well-known that in the ring of formal power series  $K[[x]]$, the invertible elements are precisely those formal power series $\gamma(x) = \sum_{i=0}^\infty k_i x^i$ for which $k_0 \neq 0$;  i.e., for which $\gamma(0)\neq 0$.

\begin{lemma}\label{extendfromRp(c)}
Let $p(x)=p_0+p_1x+\cdots+p_nx^n \in \mathcal P$.   Any $R$-homomorphism from $Rp(c)$ to $Y$ extends to an $R$-homomorphism from $R$ to $Y$.
\end{lemma}
\begin{proof}
Let $\psi: Rp(c)\to Y$ be a homomorphism of left $R$-modules.  Let $\psi(p(c))=y$, where $y=k_{-1}w+k_0d+k_1cd+\cdots+k_ic^id+\cdots$. We need to find an $R$-homomorphism $\Psi:R\to Y$ such that $$\Psi(p(c))=p(c)\Psi(1)=y.$$ 

Because $p_0 \neq 0$, viewing $p(x) \in K[[x]]$ there exists $\alpha(x) \in K[[x]]$ for which $p(x)\alpha(x) = 1$ in $K[[x]]$.   Write $\alpha(x) = a_0 + a_1x + a_2x^2 + \cdots.$ Set $p(x)=\sum_{i=0}^\infty p_ix^i$, with $p_i=0$ $\forall i>n$; then 
$$p_0a_0 = 1, \ \ \mbox{and} \ \ \  \ \sum_{j=0}^N p_ja_{N-j} = 0 \ \ \mbox{for all} \ N\geq 1 \ . \quad\quad\qquad  (\star)$$

\noindent Now define the following elements of $K$:
$$z_{-1} := a_0k_{-1}, \ \ \mbox{and, for each } M\geq 0,     \ z_M := \sum_{i=0}^M a_i k_{M-i}.$$

\noindent We construct  $z \in Y$ by setting
$$z := z_{-1}w + z_0d + z_1cd + z_2c^2d + \cdots ,$$
\noindent so that $p(c)z = (p_0 1_R + p_1c + p_2c^2 + \cdots + p_nc^n) (z_{-1}w + z_0d + z_1cd + z_2c^2d + \cdots)$. We already know that $p_0a_0 = 1$, so that $p_0z_{-1}=p_0a_0k_{-1}=k_{-1}$. Moreover, by standard computations and using the previous relations $(\star)$, one can show that for any $i\geq 0$, the coefficient of the term $c^id$ in  $p(c)z$  equals the coefficient of the term $c^id$ in $y$. This gives  that 
$p(c)z = y$ in  $Y$.   (Intuitively,  the  idea here  is to  ``define informally"  the expression $\alpha(c) =  a_01_R + a_1c + a_2c^2 + \cdots$, and subsequently the element  $z\in Y$ as $z = \alpha(c)y$, so that  $p(c)\alpha(c)y=1\cdot y=y$.)

Finally,  consider the $R$-homomorphism $\Psi:R\to Y$ obtained by setting  
$$\Psi(1)=z.$$ 
Then  
$\Psi(p(c))=p(c)\Psi(1)=p(c)z=y$, as desired.
\end{proof}


\begin{proposition}\label{Yinjective}
The left $L_K(\mathcal{T})$-module $Y$ is injective.
\end{proposition}
\begin{proof}
We again invoke Corollary \ref{simplifyBaer}, which yields that we need only test the injectivity of $Y$ with respect to the two indicated types of left $R$-ideals.  But this is precisely what has been achieved in Lemmas \ref{extendfromJ} and \ref{extendfromRp(c)}. 
\end{proof}

\begin{corollary}\label{YisenvelopeofRw}
$Y$ is the injective envelope of $Rw$.
\end{corollary}
\begin{proof}
As noted in Corollary \ref{UfisenvelopeofVf},  the injective envelope of any module is an injective module in which the given module sits as an essential submodule.   So the result follows from Lemma \ref{RwessentialinY} and Proposition \ref{Yinjective}.  
\end{proof}

We have the following description of the quotient $Y/Rw$, namely, that it is  an extension of a direct summand of a product of copies of the $U^f$'s by the simple module $Rw$.
\begin{proposition}
The module $Y/{Rw}$  is a direct summand of a product of copies of the $U^f$'s. 
\end{proposition}
\begin{proof}
Consider the short exact sequence $0\to Rw\to Y\to Y/{Rw}\to 0$. First notice that $\Hom(Rw, Y/{Rw})=0$, as follows.  To the contrary, suppose there exists   $0\neq f: Rw\to Y/{Rw}$. Then by the simplicity of $Rw$, the map $f$ must be a monomorphism.  Further, since $Rw$ is projective, there then exists $\tilde{f}:Rw\to Y$ such that $\pi\circ \tilde{f}=f$. In particular $\Im \tilde{f}\cap Rw=0$.  But this is a contradiction since $Rw$ is the essential socle of $Y$.

Now let $0\neq x\in Y/{Rw}$, and consider the cyclic module $Rx\cong R/{\Ann (x)}$. Let $M$ be a maximal  left ideal of $R$ containing $\Ann (x)$, so that $Rx\to  {R/M}\to  0$. If $R/M\cong Rw$,  since $Rw$ is projective   we would get that $Rw$ is a summand of $Rx$, in particular is a submodule of $Rx$ and thereby  also of $Y/{Rw}$, contrary to the result of the previous paragraph. So $R/M$ is a simple module of type $V^f$, and hence it embeds in $U^f$. In such a way, for any $0\neq x\in Y/{Rw}$, there is a suitable $f(x)\in \mathcal{F}$ and  a  morphism $\varphi_x:Rx\to  U^f$, such that $\varphi_x(x)\neq 0$. Since $U^f$ is injective, $\varphi_x$ extends to a morphism $\tilde{\varphi_x}: Y/{Rw}\to  U^f$. 
 So we get that  $Y/{Rw}$ embeds in a product of copies of the $U^f$ ($f(x)\in\mathcal F$).   But $Y$ is injective, and so by Corollary~\ref{quotientsofinjectives}  also $Y/Rw$ is injective, and thus $Y/Rw$ indeed is  a direct summand of the product of copies of the $U^f$'s.    \end{proof}

\bigskip

\subsection{Consequences of Subsections \ref{subsectionV^f} and \ref{subsectionRw}}

Every ring admits (up to isomorphism) a unique minimal injective cogenerator (see e.g. \cite[Section 18]{AF} for a full description of this concept). Since any  representation of the ring  embeds in a product of copies of a cogenerator, once we know such a cogenerator we can  describe the entire category of modules over the ring. Thanks to previous  results we are able to determine a minimal injective cogenerator for the  algebra $L_K(\mathcal{T})$.

\begin{theorem}\label{Yismininjcogen}
The $L_K(\mathcal{T})$-module $$C = Y \oplus (\oplus_{f(x) \in \mathcal{F}} U^f ) $$ is a minimal injective cogenerator for  $L_K(\mathcal{T})$.
\end{theorem}
\begin{proof}
By combining Proposition \ref{prop:U} with Proposition \ref{YisenvelopeofRw} we get immediately that $C$
is injective.   As noted at the start of Section \ref{simplessection}, because there is a unique cycle in $\mathcal{T}$,  \cite[Corollary 4.6]{AR} applies, and  yields that, up to isomorphism, the set of all the simple modules over $L_K(\mathcal{T})$ consists of $Rw$ together with the (pairwise non-isomorphic) modules of the form  $\{V^f \ | \ f(x) \in \mathcal{F}\}$.  
Thus $C$ 
contains a copy of every simple left $L_K(\mathcal{T})$-module, and so it is a cogenerator for the module category \cite[Proposition~18.15]{AF}. Since any injective cogenerator has to contain of copy of the injective envelope of any simple module, we get that  $C$
 is a minimal injective cogenerator for  $L_K(\mathcal{T})$.
\end{proof}

When $S$ is any Noetherian ring, then the minimal injective cogenerator is precisely the direct sum of the injective envelopes of the simple modules.  In Theorem \ref{Yismininjcogen}  we have reached the same conclusion  for the non-Noetherian ring $L_K(\mathcal{T})$, and, moreover, have described each of these injective envelopes explicitly.

With  Theorem \ref{Yismininjcogen} in hand, we achieve a description of all the injective left $R$-modules. 
\begin{corollary}\label{allinjectives}
Let $C$ denote $Y \oplus (\oplus_{f(x) \in \mathcal{F}} U^f )$.  Then a left $L_K(\mathcal{T})$-module $M$ is injective if and only if $M$ is isomorphic to a direct summand of a direct product of copies of $C$.  
\end{corollary}
\begin{proof}
This follows immediately from the definition of a cogenerator, together with the facts that  direct products and direct summands of injective modules are injective, and an injective submodule of a module is necessarily a direct summand. 
\end{proof}

%
%
%
%
%
%
%
%
%
%
%
%
%
%
%
%
%

\bigskip

\begin{center}
 \textsc{Acknowledgements}
\end{center}

Part of this work was carried out during a visit of the first author to
the University of Padova, Department of Statistical Sciences.
 The first author is pleased to take this opportunity to express his
thanks to the host institution, and its faculty, for its warm hospitality
and travel support by the project of excellence ``Statistical methods and models for complex data''
(Department of Statistical Sciences)  and the grant ``Iniziative di cooperazione
universitaria Anno 2019'' (University of Padova).


\begin{thebibliography}{00}

\bibitem{TheBook}   G. Abrams, P. Ara, M. Siles Molina.  Leavitt path algebras.  Lecture Notes in Mathematics vol. 2191.   Springer Verlag, London, 2017.  ISBN-13:  978-1-4471-7344-1.  DOI:  10.1007/978-1-4471-7344-1
 

\bibitem{AMTExt}  G. Abrams, F. Mantese, A. Tonolo, {\it Extensions of simple modules over Leavitt path algebras}, Journal of Algebra {\bf 431}  (2015), pp. 78 -- 106.  DOI: 10.1016/j.jalgebra.2015.01.034 

\bibitem{AMTBezout}   G. Abrams, F. Mantese, A. Tonolo,   \emph{Leavitt path algebras are B\'{e}zout}, Israel J. Math. {\bf 228} (2018), pp. 53--78.  DOI: 10.1007/s11856-018-1773-2

\bibitem{AMTPrufer}   G. Abrams, F. Mantese, A. Tonolo, \emph{Pr\"{u}fer modules over Leavitt path algebras}, J. Algebra App.  {\bf 18} (2019), 1950154 (28 pp.).   DOI:  10.1142/s0219498819501548


\bibitem{AF} F. Anderson and K. Fuller.    Rings and categories of modules. Second edition. Graduate Texts in Mathematics, 13. Springer-Verlag, New York, 1992.     ISBN: 0-387-97845-3.   DOI: 10.1007/978-1-4684-9913-1

\bibitem{AR}    P. Ara and K.~Rangaswamy, \emph{Finitely presented simple modules over Leavitt path algebras}, Journal of Algebra {\bf 417} (2014), pp. 333 -- 352.  DOI: 10.1016/j.jalgebra.2014.06.032

\bibitem{Bavula}  V.  V. Bavula, \emph{The algebra of one-sided inverses of a polynomial algebra}, J. Pure App. Alg. {\bf 214} (2010), pp. 1874--1897.  DOI: 10.1016/j.jpaa.2009.12.033 

 \bibitem{Bergman}  G. Bergman, \emph{Coproducts and some universal ring constructions}, Trans. A.M.S. {\bf 200} (1974), pp. 33--88.  DOI: 10.1090/s0002-9947-1974-0357503-7 
 
  \bibitem{Chen} X.\ W.\ Chen, \emph{Irreducible representations of Leavitt path algebras,} Forum Math. \textbf{27}(1)  (2015),  pp. 549--574.  DOI: 10.1515/forum-2012-0020 
 
 \bibitem{Cohn} P.M. Cohn, \emph{Some remarks on the Invariant Basis property}, Topology {\bf 5} (1966), pp 215--228.  DOI: 10.1016/0040-9383(66)90006-1
 

 \bibitem{Gerritzen} L. Gerritzen, \emph{Modules over the algebra of the noncommutative equation $yx = 1$}, Arch.
Math. {\bf 75} (2000), pp. 98--112.   DOI: 10.1007/pl00000437
 
 \bibitem{IovanovSistko} M. Iovanov and A. Sistko, \emph{On the Toeplitz-Jacobson algebra and direct finiteness}, pp. 113--124 in \emph{Groups, Rings, Group Rings, and Hopf Algebras}, Contemporary Math {\bf 668}, Amer. Math. Soc., Providence, RI, 2017.    DOI: 10.1090/conm/688/13830
 
 \bibitem{Jacobson}   N. Jacobson, \emph{Some Remarks on One-Sided Inverses}, Proc. A.M.S. {\bf 1} (1950), pp. 352--355.  DOI: 10.1007/978-1-4612-3694-8\_6
 
 \bibitem{Lam}  T.Y. Lam.  Lectures on Modules and Rings.  Graduate Texts in Mathematics vol. 189.  Springer-Verlag Berlin Heidelberg, 1999.   ISBN-13:   978-1-4612-6802-4.  DOI: 10.1007/978-1-4612-0525-8
 
 \bibitem{LuWangWang}  Z. Lu, L. Wang, X. Wang,  \emph{Nonsplit module extensions over the one-sided inverse of $k[x]$},  Involve {\bf 12}(8) (2019), pp, 1369 -- 1377.  DOI: 10.2140/involve.2019.12.1369

\bibitem{Ranga}   K.M. Rangaswamy, \emph{On simple modules over Leavitt path algebras}, J. Algebra {\bf 423}, 2015, 239-258.  DOI: 10.1016/j.jalgebra.2014.10.008 

\end{thebibliography}
\end{document}